\documentclass[10pt]{amsart}
\usepackage{graphicx}
\usepackage{amscd}
\usepackage{amsmath}
\usepackage{amsthm}
\usepackage{amsfonts}
\usepackage{amssymb}
\usepackage{mathrsfs}
\usepackage{enumerate}
\usepackage{amsrefs}
\usepackage[pagebackref,hypertexnames=false, colorlinks, citecolor=blue, linkcolor=red, pdfstartview=FitB]{hyperref}
\usepackage[latin1]{inputenc}

\newcommand{\D}{\operatorname{\mathbb{D}}}

\newcommand{\C}{\operatorname{\mathbb{C}}}
\newcommand{\T}{\operatorname{\mathbb{T}}}

\newcommand{\A}{\operatorname{\mathcal{A}}}
\newcommand{\E}{\operatorname{\mathcal{E}}}
\newcommand{\Fil}{\operatorname{\mathcal{F}}}
\newcommand{\hil}{\operatorname{\mathcal{H}}}
\newcommand{\kil}{\operatorname{\mathcal{K}}}

\newcommand{\ol}{\overline }
\let\phi\varphi
 \DeclareMathOperator{\Ext}{Ext}

\newtheorem{lemma}{Lemma}[section]
\newtheorem{theorem}[lemma]{Theorem}
\newtheorem{proposition}[lemma]{Proposition}
\newtheorem{corollary}[lemma]{Corollary}
\theoremstyle{definition}
\newtheorem{definition}[lemma]{Definition}

\begin{document}
\author{Rapha\"el Clou\^atre}
\address{Department of Mathematics, Indiana University, 831 East 3rd Street,
Bloomington, IN 47405} \email{rclouatr@indiana.edu}
\title{On the unilateral shift as a Hilbert module over the disc algebra}
\date{\today}
%\subjclass[2010]{Primary:47A45, 47A15}
%\keywords{$C_0$ operators, invariant subspaces, quasisimilarity}
\begin{abstract}
We study the unilateral shift (of arbitrary countable multiplicity)
as a Hilbert module over the disc algebra and the associated
extension groups. In relation with the problem of determining
whether this module is projective, we consider a special class of
extensions, which we call \emph{polynomial}. We show that the
subgroup of polynomial extensions of a contractive module by the
adjoint of the unilateral shift is trivial. The main tool is a
function theoretic decomposition of the Sz.-Nagy--Foias model space
for completely non-unitary contractions.
\end{abstract}
\maketitle

\section{Introduction}
In their pioneering work \cite{douglas1989},  Douglas and Paulsen
reformulated several interesting operator theoretic questions in the
language of module theory, and in doing so introduced the notion of
Hilbert modules over function algebras. This suggested the use of
cohomological methods to further the study of problems such as
commutant lifting. Naturally, the question of identifying those
Hilbert modules which are projective arose and attracted a lot of
interest. The first result in that direction was obtained by Carlson
and Clark in \cite{carlson1995}, where it was shown that a
contractive projective Hilbert module over the disc algebra $A(\D)$
must be similar to an isometric one. Soon thereafter, the same
authors along with Foias and Williams proved in \cite{carlson1994}
that isometric modules over $A(\D)$ are projective in the category
of \emph{contractive} Hilbert modules. This turns out to be an
equivalence, as was later shown by Ferguson in \cite{ferguson1997}.
In addition, the authors of \cite{carlson1994} show that unitary
modules over $A(\D)$ are projective in the larger category of
(non-necessarily contractive) Hilbert modules.

Projective Hilbert modules over $A(\D)$ are to this day still quite
mysterious. In fact, as things stand currently, unitary modules are
the only known instances of such objects. On the other hand, by the
results mentioned above a contractive projective module must be
similar to an isometric module. In view of the classical Wold-von
Neumann decomposition of an isometry, we see that the quest to
identify the contractive projective Hilbert modules over the disc
algebra is reduced to the following question: are unilateral shifts
projective? A consequence of Pisier's famous counterexample to the
Halmos conjecture (see \cite{pisier1997}) is that the answer is
negative in the case of infinite multiplicity. Whether or not things
are different for finite multiplicities is still an open problem.

We study extension groups associated to unilateral shifts viewed as
Hilbert modules over the disc algebra. With the notation established
in Section 2, our main result (Theorem \ref{t-extcontraction})
establishes the triviality of the the subgroup of elements $[X]\in
\Ext_{A(\D)}^1(T,S^*)$ such that $S^{*N}XT^N=0$ for some integer
$N\geq 0$ whenever $T$ is similar to a contraction (here $S^*$ is
the adjoint of the unilateral shift of arbitrary countable
multiplicity). In some sense, this supports the idea that the
unilateral shift is projective. However, the reader should keep in
mind that our result holds regardless of multiplicity and thus does
not capture the fact that the shift of infinite multiplicity is not
projective. The crucial ingredient for the proof of Theorem
\ref{t-extcontraction} is a decomposition of the Sz.-Nagy--Foias
model space $H(\Theta)$ which we think is of independent interest
(see Theorem \ref{t-ZEfunctmodel}).

There has been further work on the question
of projective Hilbert modules following the appearance of
\cite{carlson1995}, \cite{carlson1994} and \cite{ferguson1997}.
Generalizing the fact that unitary modules are projective over
$A(\D)$, it was shown in \cite{chen2000} that whenever the algebra
$A$ is a so-called unit modulus algebra and the module action can be
extended to an action of $C(\partial A)$ (here $\partial A$ denotes
the Shilov boundary of $A$), then the module is projective. An
earlier paper of Guo (see \cite{guo1999}) establishes using
essentially the same idea that the result holds for the ball algebra
$A(\mathbb{B}^N)$ under an additional continuity assumption on the
module action. This assumption was later removed by Didas and
Eschmeier in \cite{didas2006}, where domains more general than the
ball are considered. The case of the polydisc algebra $A(\D^N)$ was
first considered in \cite{carlson1997}, where results exhibiting a
sharp contrast with the one dimensional case were obtained. From the
point of view of reproducing kernel Hilbert spaces, Clancy and
McCullough showed in \cite{clancy1998} that the Hilbert space
$H^2(k)$ associated to a Nevanlinna-Pick kernel $k$ considered as a
Hilbert module over its multiplier algebra is projective in an
appropriate category.  The existence of a projective Hilbert module
over very general function algebras was established in
\cite{gulinskiy2003}. Note finally that the notion of Hilbert
modules and the question of projectivity have also been studied over
general operator algebras, see \cite{muhly1995}.

The paper is organized as follows. Section 2 introduces the
necessary preliminaries about Hilbert modules. In Section 3 we
develop some technical tools which are used in Section 5 to obtain
the main result. In the meantime, we examine in Section 4 some
simple examples and offer some explicit calculations of the objects
introduced in Section 3. Finally, in Section 6 we briefly address
the issue of non-contractive modules by considering operators of the
type constructed by Pisier in \cite{pisier1997}.

\section{Preliminaries}
Let $\hil$ be a Hilbert space and let $T:\hil \to \hil$ be a bounded
linear operator, which we indicate by $T\in B(\hil)$. Recall that
the operator $T$ is said to be \emph{polynomially bounded} if there
exists a constant $C>0$ such that for every polynomial $\phi$, we
have
$$
\|\phi(T)\|\leq C\|\phi\|_{\infty}
$$
where
$$
\|\phi\|_{\infty}=\sup_{|z|<1}|\phi(z)|.
$$
This inequality allows one to extend continuously the polynomial
functional calculus $\phi\mapsto \phi(T)$ to all functions $\phi$ in
the disc algebra $A(\D)$, which consists of the holomorphic
functions on $\D$ that are continuous on $\ol{\D}$ (throughout the
paper $\D$ denotes the open unit disc and $\T$ denotes the unit
circle).

If $T\in B(\hil)$ is a polynomially bounded operator, the map
$$
A(\D)\times \hil\to \hil
$$
$$
(\phi, h)\mapsto \phi(T)h
$$
gives rise to a structure of an $A(\D)$-module on $\hil$, and we say
that $(\hil, T)$ is a \emph{Hilbert module} (see \cite{douglas1989}
for more details). We only deal with $A(\D)$-modules in this paper,
so no confusion may arise regarding the underlying function algebra
and we usually do not mention it explicitely. Moreover, when the
underlying Hilbert space is understood, we slightly abuse
terminology and say that $T$ is a Hilbert module.

Given two Hilbert modules $(\hil_1,T_1)$ and $(\hil_2,T_2)$,  we can
consider the extension group $\Ext_{A(\D)}^1(T_2,T_1)$. This group
consists of equivalence classes of exact sequences
$$
0 \to \hil_1\to \kil\to \hil_2\to 0
$$
where $\kil$ is another Hilbert module and each map is a module
morphism. Rather than formally defining the equivalence relation and
the group operation, we simply use the following characterization
from \cite{carlson1995}.
\begin{theorem}\label{t-extchar}
Let $(\hil_1, T_1)$ and $(\hil_2, T_2)$ be Hilbert modules. Then,
the group
$$
\Ext_{A(\D)}^1(T_2,T_1)
$$
is isomorphic to $\A/\mathcal{J}$, where $\A$ is the space of
operators $X:\hil_2\to \hil_1$ for which the operator
$$
\left(
\begin{array}{cc}
T_1 & X \\
0 & T_2
\end{array}
\right)
$$
is polynomially bounded, and $\mathcal{J}$ is the space of operators
of the form $T_1L-LT_2$ for some bounded operator $L:\hil_2\to
\hil_1$.
\end{theorem}

Extension groups are invariant under similarity: if $(\hil_1',T_1')$
and $(\hil_2',T_2')$ are Hilbert modules which are similar to
$(\hil_1,T_1)$ and $(\hil_2,T_2)$ respectively, then the groups
$\Ext_{A(\D)}^1(T_2,T_1)$ and $\Ext_{A(\D)}^1(T_2',T_1')$ are
isomorphic.

In view of Theorem \ref{t-extchar}, the next lemma is useful. Before
stating it, we recall a well-known estimate. Let
$$
D:A(\D)\to A(\D)
$$
be defined as
$$
(Df)(z)=\frac{1}{z}(f(z)-f(0))
$$
for every $z\in \D$ and $f\in A(\D)$. It is a classical fact that
there exists a constant $M>0$ such that
$$
\|D^n\|\leq M(1+\log n) $$ for every $n\geq 1$.

\begin{lemma}\label{l-finitebounded}
Let $(\hil_1, T_1)$ and $(\hil_2, T_2)$ be Hilbert modules. Let
$X:\hil_2 \to \hil_1$ be a bounded operator such that
$T_1^NXT_2^N=0$ for some integer $N\geq 0$. Then, the operator
$R:\hil_1\oplus \hil_2\to \hil_1\oplus \hil_2$ defined as
$$
R=\left(
\begin{array}{cc}
T_1 & X\\
0 & T_2
\end{array}
\right)
$$
is polynomially bounded.
\end{lemma}
\begin{proof}
Choose a $\phi(z)=\sum_{k=0}^d a_k z^k$. A quick
calculation shows that
$$
\phi(R)= \left(
\begin{array}{cc}
\phi(T_1) & \delta_X(\phi)\\
0 & \phi(T_2)
\end{array}
\right)
$$
where
$$
\delta_X(\phi)=\sum_{k=1}^da_k \sum_{j=0}^{k-1}T_1^j XT_2^{k-1-j}.
$$
Since both $T_1$ and $T_2$ are polynomially bounded, there exist
constants $C_1,C_2>0$ independent of $\phi$ such that
$$
\|\phi(T_1)\|\leq C_1\|\phi\|_{\infty}
$$
and
$$
\|\phi(T_2)\|\leq C_2\|\phi\|_{\infty}.
$$
Therefore, we simply need to verify that there exists a constant
$C>0$ independent of $\phi$ such that
$$
\|\delta_X(\phi)\|\leq C\|\phi\|_{\infty}.
$$
We have
\begin{align*}
\|\delta_X(\phi)\|&\leq \sum_{k=1}^{d} |a_k|
\sum_{j=0}^{k-1}\|T_1^jXT_2^{k-1-j}\|\\
&\leq C_1 C_2 \|X\| \sum_{k=1}^{d} k|a_k|\\
&= C_1 C_2 \|X\| \sum_{k=1}^{d} k\left|
\frac{\phi^{(k)}(0)}{k!}\right|\\
\end{align*}
and in light of the classical Cauchy estimates, we find
\begin{equation}\label{e-deltanorm}
\|\delta_X(\phi)\|\leq C_1 C_2
\frac{d(d+1)}{2}\|X\|\|\phi\|_{\infty}.
\end{equation}
In particular, if we set
$$
C=C_1 C_2 N(2N-1)\|X\|,
$$
which depends only on $X,T_1,T_2$ and $N$, then
$$
\|\delta_X(\phi)\|\leq C \|\phi\|_{\infty}
$$
whenever $\phi$ has degree at most $2N-1$. We focus therefore on the
case where $d\geq 2N$. We have
\begin{equation}\label{e-delta}
\delta_X(\phi)= \sum_{k=1}^{2N-1} a_k
\sum_{j=0}^{k-1}T_1^jXT_2^{k-1-j}+\sum_{k=2N}^{d} a_k
\sum_{j=0}^{k-1}T_1^jXT_2^{k-1-j}.
\end{equation}
Since
$$
\left\|\sum_{k=1}^{2N-1} a_k
\sum_{j=0}^{k-1}T_1^jXT_2^{k-1-j}\right\|\leq C \|\phi\|_{\infty}
$$
we are left with estimating the second sum in (\ref{e-delta}), where
$k\geq 2N$. By assumption, we know that $T_1^jXT_2^{k-1-j}\neq 0$
only when $j\leq N-1$ or $k-1-j\leq N-1$. This allows us to write
\begin{align*}
\sum_{k=2N}^{d} a_k \sum_{j=0}^{k-1}T_1^jXT_2^{k-1-j} &=\sum_{k=2N}^{d} a_k \sum_{j=0}^{N-1}T_1^jXT_2^{k-1-j}+\sum_{k=2N}^{d} a_k \sum_{j=k-N}^{k-1}T_1^jXT_2^{k-1-j}\\
&=\sum_{k=2N}^{d} a_k \sum_{j=0}^{N-1}T_1^jXT_2^{k-1-j}+\sum_{k=2N}^{d} a_k \sum_{j=0}^{N-1}T_1^{k-1-j}XT_2^{j}\\
&=\sum_{j=0}^{N-1} (T_1^j X\Phi_j(T_2)+\Phi_j(T_1)XT_2^j)
\end{align*}
where
$$
\Phi_j(z)=\sum_{k=2N}^d a_k z^{k-1-j}
$$
whenever $0\leq j \leq N-1$. Notice now that
$$
\Phi_j(z)+\sum_{k=j+1}^{2N-1} a_k z^{k-1-j}= (D^{j+1} \phi)(z)
$$
whence
\begin{align*}
\|\Phi_j\|_{\infty}&\leq \| D^{j+1} \phi\|_{\infty}+
\sum_{k=0}^{2N-1}|a_k|\\
&= \| D^{j+1} \phi\|_{\infty}+ \sum_{k=0}^{2N-1}\left|
\frac{\phi^{(k)}(0)}{k!}\right|
\end{align*}
for every  $0\leq j \leq N-1$. Another use of the Cauchy estimates
along with the remark preceding the statement of the lemma implies
the existence of a constant $C'>0$ depending only on $N$ such that
$$
\|\Phi_j\|_{\infty}\leq C' \|\phi\|_{\infty}
$$
for every $0\leq j \leq N-1$. Thus,
$$
\left\|\sum_{k=2N}^{d} a_k
\sum_{j=0}^{k-1}T_1^jXT_2^{k-1-j}\right\|\leq 2 N C'C_1 C_2 \|X\|
\|\phi\|_{\infty}
$$
and
\begin{align*}
\|\delta_X(\phi)\|&\leq \left\|\sum_{k=0}^{2N-1} a_k \sum_{j=0}^{k-1}T_1^jXT_2^{k-1-j}\right\|+\left\|\sum_{k=2N}^{d} a_k \sum_{j=0}^{k-1}T_1^jXT_2^{k-1-j}\right\|\\
&\leq C''\|\phi\|_{\infty}
\end{align*}
where $C''>0$ depends only on $N, X, T_1, T_2$. The proof is
complete.
\end{proof}

An important question in the study of extension groups is that of
determining which Hilbert modules $(\hil_2,T_2)$ have the property
that
$$
\Ext_{A(\D)}^1(T_2,T_1)=0
$$
for every Hilbert module $(\hil_1,T_1)$.  Such Hilbert modules are
said to be \emph{projective}. It is easy to verify using Theorem
\ref{t-extchar} that the map $[X]\mapsto [X^*]$ establishes an
isomorphism between the groups $\Ext^1_{A(\D)}(T_2,T_1)$ and
$\Ext^1_{A(\D)}(T_1^*,T_2^*)$, so $T_2$ is projective if and only if
$$
\Ext_{A(\D)}^1(T_1,T_2^*)=0
$$
for every Hilbert module $(\hil_1,T_1)$. A characterization of
projective Hilbert modules has long been sought. This result from
\cite{carlson1994} was mentioned in the introduction.

\begin{theorem}\label{thmunitarysplit}
If $T\in B(\hil)$ is similar to a unitary operator, then the Hilbert
module $(\hil,T)$ is projective.
\end{theorem}

If $\E$ is a separable Hilbert space, we denote by $L^2(\E)$ the
Hilbert space of weakly measurable square integrable functions
$f:\T\to \E$. The Hardy space $H^2(\E)$ is the closed subspace of
$L^2(\E)$ consisting of functions with vanishing negative Fourier
coefficients. Elements of $H^2(\E)$ can also be viewed as
$\E$-valued functions holomorphic on $\D$ with square summable
Taylor coefficients. We embed $\E$ in $H^2(\E)$ as the subspace
consisting of constant functions, and we denote by $P_{\E}$ the
orthogonal projection of $H^2(\E)$ onto $\E$. When $\E=\C$, we
simply write $H^2(\E)=H^2$ and $L^2(\E)=L^2$. The \emph{unilateral
shift} operator
$$
S_{\E}:H^2(\E)\to H^2(\E)
$$
is defined as
$$
(S_{\E}f)(z)=zf(z)
$$
for every $f\in H^2(\E)$. Recall that the \emph{multiplicity} of
$S_{\E}$ is the dimension of $\E$. Note also that since $S_{\E}$ is
isometric, it gives rise to a Hilbert module structure on $H^2(\E)$.

We now give a rather precise description of the group
$\Ext_{A(\D)}^1(T,S)$ where $(\hil,T)$ is any Hilbert module. This
result was originally proved in \cite{carlson1995} (Proposition
3.1.1 and Theorem 3.2.1) for the shift of multiplicity one. However,
a quick glance at the proof of Proposition 3.1.1 shows that it can
be adapted to any multiplicity, while the more general version of
Theorem 3.2.1 can be found in Lemma 2.1 of \cite{carlson1997}.

\begin{theorem}\label{t-extscharact}
Let $(\hil,T$) be a Hilbert module. Then, an operator $X:\hil\to \E$
gives rise to an element $[X]\in \Ext_{A(\D)}^1(T,S_{\E})$ if and
only if there exists a constant $c>0$ such that
$$
\sum_{n=0}^\infty \|XT^n h\|^2\leq c\|h\|^2
$$
for every $h\in \hil$. Moreover, for every $[X]\in
\Ext_{A(\D)}^1(T,S_{\E})$ there exists an operator $Y:\hil\to \E$
with the property that $[X]=[Y]$.
\end{theorem}

We bring the reader's attention to the fact that the group
$\Ext_{A(\D)}^1(T,S_{\E})$ is really of a ``scalar" nature: it
consists of elements $[X]$ where the operator $X:\hil\to H^2(\E)$
has range contained in the constant functions $\E$. We use Theorem
\ref{t-extscharact} throughout as a basis for comparison with our
own results about $\Ext_{A(\D)}^1(T,S_{\E}^*)$.

Finally, we end this section with a theorem that identifies the
projective modules in the smaller category of \emph{contractive}
Hilbert modules (see \cite{ferguson1997}).

\begin{theorem}\label{t-isomproj}
Let $T\in B(\hil)$ be similar to a contraction. The following
statements are equivalent:
\begin{enumerate}
\item[\rm{(i)}] $\Ext_{A(\D)}^1(T,S_{\E})=0$ for some separable
Hilbert space $\E$
\item[\rm{(ii)}] the Hilbert module $(\hil,T)$ is projective in
the category of Hilbert modules similar to a contractive one
\item[\rm{(iii)}] the operator $T$ is similar to an isometry.
\end{enumerate}
\end{theorem}

\section{A criterion for the projectivity of isometric Hilbert modules}

Throughout the paper we will assume that $\E$ is a separable Hilbert
space. The first result of this section is elementary. We record it
here for convenience.

\begin{lemma}\label{l-commutator}
Let $X,\Lambda:\hil\to H^2(\E)$ be bounded operators defined as
$$
Xh=\sum_{n=0}^\infty  z^n X_n h
$$
and
$$
\Lambda h=\sum_{n=0}^\infty  z^n L_n h
$$
for every $h\in \hil$, where $L_n,X_n\in B(\hil,\E)$ for every
$n\geq 0$. Then, $X=S_{\E}^*\Lambda-\Lambda T$ if and only if
$X_n=L_{n+1}-L_n T$ for every $n\geq 0$.
\end{lemma}

The following observation lies at the base of our investigations.
\begin{lemma}\label{l-innerchar}
Let $(\hil, T)$ be a Hilbert module. Let $X:\hil\to H^2(\E)$ be a
bounded operator defined as
$$
Xh=\sum_{n=0}^\infty z^n X_n h
$$
for every $h\in \hil$, where $X_n\in B(\hil,\E)$ for every $n\geq
0$. Let $c>0$ and $L\in B(\hil,\E)$. Then, there exists a bounded
operator $\Lambda:\hil\to H^2(\E)$ such that
$$
X=S_{\E}^*\Lambda-\Lambda T,
$$
$$
P_{\E}\Lambda=-L
$$
and
$$
\|\Lambda h\|^2\leq \|L h\|^2+ c\|h\|^2
$$
for every $h\in \hil$ if and only if
$$
\sum_{n=1}^\infty
\left\|\left(\sum_{j=0}^{n-1}X_{n-1-j}T^j-LT^n\right)h
\right\|^2\leq c\|h\|^2
$$
for every $h\in \hil$.
\end{lemma}
\begin{proof}
Assume first that
$$ \sum_{n=1}^\infty
\left\|\left(\sum_{j=0}^{n-1}X_{n-1-j}T^j-LT^n\right)h
\right\|^2\leq c\|h\|^2
$$
for every $h\in \hil$. Set $L_0=-L$ and
$$
L_n=\sum_{j=0}^{n-1}X_{n-1-j}T^{j}-LT^n
$$
for $n\geq 1$. Notice now that we have
$$
L_{0}T=-LT=L_1-X_0
$$
and
\begin{align*}
L_n T&=\sum_{j=0}^{n-1}X_{n-1-j}T^{j+1}-LT^{n+1}\\
&=\sum_{j=1}^{n}X_{n-j}T^j-L T^{n+1}\\
&=\sum_{j=0}^{n}X_{n-j}T^j-X_n-L T^{n+1}\\
&=L_{n+1}-X_n
\end{align*}
for $n\geq 1$, which shows that for every $n\geq 0$ we have
$$
X_n=L_{n+1} -L_{n}T.
$$
Define
$$
\Lambda h=\sum_{n=0}^\infty z^n L_n h
$$
for every $h\in \hil$. By Lemma \ref{l-commutator}, we see that
$$
S_{\E}^*\Lambda-\Lambda T=X.
$$
Moreover, by assumption we have for every $h\in \hil$ that
\begin{align*}
\|\Lambda h\|^2&=\sum_{n=0}^\infty \|L_n h \|^2\\
&=\|L h\|^2+\sum_{n=1}^\infty \left\|\left( \sum_{j=0}^{n-1}X_{n-1-j}T^j-LT^n\right)h\right\|^2\\
&\leq \|Lh \|^2 +c\|h\|^2.
\end{align*}

Conversely, assume that there exists a bounded linear operator
$\Lambda :\hil \to H^2(\E)$ defined as
$$
\Lambda h=\sum_{n=0}^\infty z^n L_n h
$$
for every $h\in \hil$ with the property that
$$
X=S_{\E}^*\Lambda-\Lambda T,
$$
$$
L_0=-L
$$
and
$$
\|\Lambda h\|^2\leq \|Lh\|^2+c\|h\|^2
$$
for every $h\in \hil.$ Then, by Lemma \ref{l-commutator} we have that
$$
X_n=L_{n+1}-L_n T
$$
for every $n\geq 0$, so that we find
\begin{align*}
\sum_{j=0}^{n}X_{n-j}T^j&=\sum_{j=0}^{n}(L_{n-j+1}-L_{n-j}T)T^j\\
&=\sum_{j=0}^{n}L_{n-j+1}T^j-\sum_{j=1}^{{n+1}}L_{n-j+1}T^j\\
&=L_{n+1}-L_{0}T^{n+1}\\
&=L_{n+1}+LT^{n+1}
\end{align*}
for every $n\geq 0$. Consequently,
$$
\sum_{n=1}^\infty \left\|\left(L
T^n-\sum_{j=0}^{n-1}X_{n-1-j}T^j\right)h \right\|^2
=\sum_{n=1}^\infty \|L_n h\|^2 = \|\Lambda h\|^2-\|L h\|^2\leq
c\|h\|^2
$$
and the proof is complete.
\end{proof}
As suggested by this result, we make the following definition.
\begin{definition}\label{d-z}
Let $(\hil,T)$ be a Hilbert module and let $\E$ be a separable Hilbert space. We denote by $Z_{\E}(T)$ the subspace of $B(\hil,\E)$ consisting of the operators $X\in
B(\hil,\E)$ with the property that there exists a constant $c_X>0$
such that
$$
\sum_{n=0}^\infty \|XT^n h \|^2\leq c_X\|h\|^2
$$
for every $h\in \hil$.
\end{definition}
By Theorem \ref{t-extscharact}, we see that
the set  $Z_{\E}(T)$ consists exactly of those operators $X:\hil\to
\E$ which give rise to an element $[X]\in \Ext_{A(\D)}^1(T,S_{\E})$.

We now give a criterion for an element $[X]$ of
$\Ext_{A(\D)}^1(T,S_{\E}^*)$ to be trivial when $X$ is particularly
simple, namely of the type considered  in Lemma
\ref{l-finitebounded}.

\begin{theorem}\label{t-innercrit}
Let $(\hil, T)$ be a Hilbert module. Let $X:\hil\to H^2(\E)$ be
defined as
$$
Xh=\sum_{n=0}^\infty  z^n X_n h
$$
for every $h\in \hil$, where $X_n\in B(\hil,\E)$ for every $n\geq
0$. Assume that $S_{\E}^{*N}XT^N=0$ for some integer $N\geq 0$.
Then, the element $[X]$ of $\Ext_{A(\D)}^1(T,S_{\E}^*)$ is trivial
if and only if
$$
\sum_{j=0}^{N-1}X_{j}T^{N-1-j}\in B(\hil,\E)T^{N}+Z_{\E}(T).
$$
\end{theorem}
\begin{proof}
By Lemma \ref{l-innerchar}, we find that $[X]=0$ if and only if for
some $L\in B(\hil,\E)$ we have
$$
\sum_{n=1}^\infty \left\|\left( LT^n
+\sum_{j=0}^{n-1}X_{n-1-j}T^j\right)h\right\|^2\leq c\|h\|^2
$$
for some constant $c>0$ and every $h\in \hil$. Now, the condition
$S_{\E}^{*N}XT^N=0$ implies that $X_{n}T^m=0$ if $n\geq N$ and
$m\geq N$. Thus, $X_{n-1-j}T^j\neq 0$ only if $j\leq N-1$ or $j\geq
n-N$. Therefore, for $n\geq 2N$, we can write
\begin{align*}
\sum_{j=0}^{n-1}X_{n-1-j}T^j &=\sum_{j=0}^{N-1}X_{n-1-j}T^j+\sum_{j=n-N}^{n-1}X_{n-1-j}T^j\\
&=\sum_{j=0}^{N-1}(X_{n-1-j}T^j+X_{j}T^{n-1-j}).
\end{align*}
Notice now that
\begin{align*}
\sum_{n=2N}^\infty \left\|\left(\sum_{j=0}^{N-1}X_{n-1-j}T^j\right)h\right\|^2
&\leq N\sum_{n=2N}^\infty \sum_{j=0}^{N-1}\left\|X_{n-1-j}T^j h\right\|^2\\
&\leq N \sum_{j=0}^{N-1} \|X T^{j}h\|^2\\
&\leq N^2 \|X\|^2 C_T^2\|h\|^2
\end{align*}
where as usual $C_T>0$ is a constant satisfying
$$
\|\phi(T)\|\leq C_T\|\phi\|_{\infty}
$$
for every $\phi\in A(\D)$. Thus, $[X]=0$ if and only if
$$
\sum_{n=2N}^\infty \left\|\left( LT^n
+\sum_{j=0}^{N-1}X_{j}T^{n-1-j}\right)h\right\|^2\leq c\|h\|^2
$$
which is in turn equivalent to
$$
\sum_{n=2N}^\infty \left\|\left( LT^N
+\sum_{j=0}^{N-1}X_{j}T^{N-1-j}\right)T^{n-N}h\right\|^2\leq
c\|h\|^2
$$
and thus to
$$
LT^{N}+\sum_{j=0}^{N-1}X_{j}T^{N-1-j}\in
Z_{\E}(T).
$$
\end{proof}

\begin{definition}\label{d-poly}
Let $\E$ be a separable Hilbert space. Given two Hilbert modules $(H^2(\E),T_1)$ and $(\hil,T_2)$, we
define the \emph{polynomial subgroup} $\Ext^1_{\rm{poly}}(T_2,T_1)$
of $\Ext^1_{A(\D)}(T_2,T_1)$  to be the subgroup of elements $[X]$
such that $S_{\E}^{*N}XT_2^N=0$ for some integer $N\geq 0$.
\end{definition}
We are primarily interested in the case of $T_1=S_{\E}$ or $T_1=S_{\E}^*$.
In particular, we obtain the following consequence of Theorem
\ref{t-innercrit}.

\begin{corollary}\label{c-ext0crit}
Let $S_{\E}:H^2(\E)\to H^2(\E)$ be the unilateral shift and let
$(\hil,T)$ be a Hilbert module. Then
$$
B(\hil,\E)T+Z_{\E}(T)=B(\hil,\E)
$$
if and only if
$$
\Ext^1_{\rm{poly}}(T,S_{\E}^*)=0.
$$
\end{corollary}
\begin{proof}
Note that $Z_{\E}(T)T\subset Z_{\E}(T)$. Thus, if
$$
B(\hil,\E)T+Z_{\E}(T)=B(\hil,\E)
$$
then by using an iterative argument we find
$$
B(\hil,\E)T^N+Z_{\E}(T)=B(\hil,\E)
$$
for every $N\geq 0$, and Theorem \ref{t-innercrit} immediately
implies that
$$
 \Ext^1_{\rm{poly}}(T,S_{\E}^*)=0.
$$
Conversely, assume the polynomial subgroup vanishes and fix $X\in
B(\hil,\E)$. In light of the equality $S^*_{\E}X=0$, Lemma
\ref{l-finitebounded} implies that the operator $X:\hil \to H^2(\E)$
gives rise to an element $[X]$ in $\Ext^1_{\rm{poly}}(T,S_{\E}^*)$
and by Theorem \ref{t-innercrit} we find
$$
X\in B(\hil,\E)T+Z_{\E}(T).
$$
Since $X\in B(\hil,\E)$ was arbitrary, we see that
$$
B(\hil,\E)T+Z_{\E}(T)=B(\hil,\E).
$$
\end{proof}

\begin{corollary}\label{c-bddbelow}
Let $S_{\E}:H^2(\E)\to H^2(\E)$ be the unilateral shift and let
$(\hil,T)$ be a Hilbert module. If
$$
\Ext^1_{\rm{poly}}(T,S_{\E})=\Ext_{\rm{poly}}^1(T,S_{\E}^*)=0,
$$
then $T$ is bounded below. Conversely, if $T$ is bounded below, then
$$
\Ext_{\rm{poly}}^1(T,S_{\E}^*)=0.
$$
\end{corollary}
\begin{proof}
Assume first that $T$ is bounded below. Then, $T$ is
left-invertible so that
$$
B(\hil,\E)T=B(\hil,\E)
$$
and we obtain
$$
\Ext_{\rm{poly}}^1(T,S_{\E}^*)=0
$$
by Corollary \ref{c-ext0crit}.

Conversely, assume
$$
\Ext^1_{\rm{poly}}(T,S_{\E})=\Ext_{\rm{poly}}^1(T,S_{\E}^*)=0.
$$
By Theorem \ref{t-extscharact} we know that $X\in Z_{\E}(T)$ if and
only if $[X]\in \Ext_{\rm{poly}}^1(T,S_{\E})$. Since this group is
assumed to be trivial, we see $X\in Z_{\E}(T)$ implies
$X=S_{\E}L-LT$. Now, the range of $X$ lies in $\E$, so we obtain
$X=-P_{\E}LT$, whence $Z_{\E}(T)\subset B(\hil,\E)T$. Using
$\Ext_{\rm{poly}}^1(T,S_{\E}^*)=0$, Corollary \ref{c-ext0crit}
implies
$$
B(\hil,\E)=B(\hil,\E)T+Z_{\E}(T)\subset B(\hil,\E)T
$$
and thus $T$  is bounded below.
\end{proof}

It is known that the fact that $T$ is bounded below isn't sufficient
for the group $\Ext_{A(\D)}^1(T,S_{\E})$ to vanish, so that the
preceding corollary cannot be improved to an equivalence. In fact,
in the case where $T$ is a contraction, the vanishing of this
extension group is equivalent to the operator $T$ being similar to
an isometry by Theorem \ref{t-isomproj}.

We obtain another consequence of Corollary \ref{c-ext0crit}, which
applies in particular to self-adjoint contractions with closed
range.

\begin{corollary}\label{c-ext0sa}
Let $S_{\E}:H^2(\E)\to H^2(\E)$ be the unilateral shift and let
$(\hil,T)$ be a Hilbert module. If
$$
T\hil\subset (\ker T)^\perp\subset T^*\hil
$$
then
$$
\Ext^1_{\rm{poly}}(T,S_{\E}^*)=0.
$$
\end{corollary}
\begin{proof}
Let $X\in B(\hil,\E)$ and set $X_1=XP_{(\ker T)^\perp}$ and
$X_2=XP_{\ker T}$ (here $P_M$ denotes the orthogonal projection onto
the closed subspace $M\subset \hil$). Then, we have that the range
of $X_1^*$ is contained in $(\ker T)^\perp \subset T^*\hil$, and
thus $X_1\in B(\hil,\E)T$ by Douglas's lemma. Moreover, since
$T\hil\subset (\ker T)^\perp $, we get $X_2 T^k=0$ for every $k\geq
1$. Therefore, $X_2\in Z_{\E}(T)$. The decomposition $X=X_1+X_2$
shows that
$$
B(\hil,\E)=B(\hil,\E)T+Z_{\E}(T)
$$
and an application of Corollary \ref{c-ext0crit} finishes the proof.
\end{proof}

We close this section with an example. In view of Corollary
\ref{c-ext0crit}, one way of establishing that $(H^2(\E),S_{\E})$ is
not a projective Hilbert module would be to find a polynomially
bounded operator $T\in B(\hil)$ that satisfies
$$
B(\hil,\E)T+Z_{\E}(T)\neq B(\hil,\E).
$$
It is easy to exhibit a polynomially bounded operator that satisfies
the weaker condition
$$
B(\hil,\E)T\cup Z_{\E}(T)\neq B(\hil,\E).
$$
Let $\E=\C$, $\hil=H^2\oplus L^2$ and $T=S_{\C}^*\oplus U$ where $U$
is the unitary operator of multiplication by the variable $e^{it}$
on $L^2$. Define $X_1:H^2\to \C$ as
$$
X_1 h=h(0)
$$
for every $h\in H^2$. Choose $\xi\in L^2$ to be a positive
function with the property that $\xi^2\notin L^2$ and define
$X_2: L^2\to \C$ as
$$
X_2 f=\langle f,\xi\rangle
$$
for every $f\in L^2$. Set $X(h_1\oplus h_2)=X_1 h_1+X_2 h_2$. We
have that $X_1 1\neq 0$, whence $X$ does not vanish on $\ker T$ and
$X\notin B(\hil,\E)T$. Moreover, the sum $\sum_{n=0}^\infty \|X_2
U^n \xi\|^2$ is infinite since
$$
\sum_{n=0}^\infty \|X_2 U^n \xi\|^2=\sum_{n=0}^\infty |\langle
\xi^2, e^{int}\rangle|^2\geq
\frac{1}{2}\sum_{n=-\infty}^{\infty}|\langle \xi^2,
e^{int}\rangle|^2
$$
and $\xi^2\notin L^2$. Therefore, $X\notin Z_{\E}(T)$ and we
conclude that
$$
X\notin B(\hil,\E)T\cup Z_{\E}(T).
$$
In particular, $[X]$ defines an ``almost" non-trivial element of
$\text{Ext}_{A(\D)}^1(T,S_{\C}^*)$.  Indeed, since $S_{\C}^* X=0$ we
have that $X$ gives rise to an element $[X]$ of
$\Ext_{A(\D)}^1(T,S_{\C}^*)$ by Lemma \ref{l-finitebounded}. Assume
that $\Lambda: \hil\to H^2$ satisfies $\Lambda \hil\subset zH^2$. We
see via Lemma \ref{l-innerchar} that
$$
X=S_{\C}^*\Lambda-\Lambda T
$$
is equivalent to $X\in Z_{\E}(T)$. Therefore,
$$
X\neq S_{\C}^*\Lambda-\Lambda T
$$
whenever $\Lambda \hil\subset zH^2$. Of course, this does not imply
that $[X]$ is a non-trivial element of $\Ext_{A(\D)}^1(T,S_{\C}^*)$,
but it shows that the equation
$$
X= S_{\C}^*\Lambda-\Lambda T
$$
has no solution if we impose the extra condition that the range of
$\Lambda $ should lie entirely in the codimension one subspace
$zH^2$.

\section{Explicit calculations of the subspace Z}
The goal of this section is to identify the spaces
$Z_{\E}(S^*_{\Fil})$ and $Z_{\E}(S_{\Fil})$ for separable Hilbert
spaces $\E, \Fil$ (recall Definition \ref{d-z}). First we set up
some notation. Given a bounded operator $X:H^2(\Fil)\to \E$, we can
write
$$
X^* e =\sum_{n=0}^\infty z^n X_n^*e
$$
for every $e\in \E$, where $X_n^*: \E\to \Fil$ for each $n$. In
particular, $X_n:\Fil\to \E$ and
$$
X h =\sum_{n=0}^\infty X_n \widehat{h}(n)
$$
where $\widehat{h}(n)\in \Fil$ denotes the $n$-th Fourier
coefficient of the function $h\in H^2(\Fil)$. Associated to $X$,
there is the Toeplitz operator
$$
T_{X}:H^2(\Fil)\to H^2(\E)
$$
defined as
$$
T_{X}h=\sum_{n=0}^\infty z^n \left(\sum_{m=n}^\infty
X_{m-n}\widehat{h}(m)\right)
$$
and the Hankel operator
$$
H_{X}: H^2(\Fil)\to  H^2(\E)
$$
defined as
$$
H_{X} h=\sum_{n=0}^\infty z^n \left(\sum_{m=0}^\infty
X_{m+n}\widehat{h}(m)\right).
$$
Typically, $T_{X}$ and $H_{X}$ are unbounded operators, but they are
always defined on the dense subset of polynomials.

\begin{proposition}\label{p-zss*}
If $S_{\Fil}:H^2(\Fil)\to H^2(\Fil)$ is the unilateral shift, then
$Z_{\E}(S_{\Fil}^*)$ consists of the operators $X\in
B(H^2(\Fil),\E)$ with the property that  $T_{X}$ is bounded, while
$Z_{\E}(S_{\Fil})$ consists of the operators $X\in B(H^2(\Fil),\E)$
with the property that  $H_{X}$ is bounded.
\end{proposition}
\begin{proof}
We first observe that
\begin{align*}
\sum_{n=0}^\infty \|XS_{\Fil}^{*n} h\|^2&=\sum_{n=0}^\infty
\left\|\sum_{m=0}^\infty X_m \widehat{h}(m+n)\right\|^2\\
&=\sum_{n=0}^\infty \left\|\sum_{m=n}^\infty X_{m-n} \widehat{h}(m)\right\|^2\\
&=\|T_Xh\|^2
\end{align*}
and
\begin{align*}
\sum_{n=0}^\infty \|XS_{\Fil}^{n} h\|^2&=\sum_{n=0}^\infty
\left\|\sum_{m=n}^\infty X_m \widehat{h}(m-n)\right\|^2\\
&=\sum_{n=0}^\infty \left\|\sum_{m=0}^\infty X_{m+n} \widehat{h}(m)\right\|^2\\
&=\|H_Xh\|^2.
\end{align*}
The result now follows directly from the definition of the spaces
$Z_{\E}(S_{\Fil}^*)$ and $Z_{\E}(S_{\Fil})$.
\end{proof}
It is well-known (see Chapter 5 of \cite{bercovici1988}) that $T_X$
is bounded if and only if the function
$$
\Phi_X(z)=\sum_{n=0}^\infty z^n X_n
$$
belongs to $H^\infty(B(\Fil,\E))$, the space of weakly holomorphic
bounded functions on $\D$ with values in $B(\Fil,\E)$. Furthermore,
$H_X$ is bounded if and only if we can find for every integer $n<0$
an operator $X_n:\Fil\to \E$ with the property that the function
$$
\sum_{n=-\infty}^\infty z^n X_n
$$
belongs to $L^\infty(B(\Fil,\E))$, the space of essentially bounded
weakly measurable functions from $\T$ into $B(\Fil,\E)$ (this is
usually referred to as the Nehari-Page theorem, see
\cite{page1970}).

In light of these remarks, let us examine what Proposition
\ref{p-zss*} says when $\E=\C$. In this case, any operator $X\in
B(H^2(\Fil),\E)$ acts as $Xh=\langle h,\xi\rangle$ for some fixed
$\xi\in H^2(\Fil)$ and thus
$$
X_n\widehat{h}(n)=\langle \widehat{h}(n),\widehat{\xi}(n)\rangle.
$$
We find
\begin{align*}
T_X h&=\sum_{n=0}^\infty \left(\sum_{m=n}^\infty \langle
\widehat{h}(m),\widehat{\xi}(m-n)\rangle\right)z^n
\end{align*}
and
\begin{align*}
H_{X}h&= \sum_{n=0}^\infty \left(\sum_{m=0}^\infty
\langle\widehat{h}(m),\widehat{\xi}(m+n)\rangle\right)z^n
\end{align*}
for every $h\in H^2(\Fil)$. This last equality shows that Corollary
3.1.6 in \cite{carlson1995} follows from Proposition \ref{p-zss*}
upon taking $\E=\C$. Of course, this is to be expected since $X\in
Z_{\C}(S_{\Fil})$ is equivalent to the fact that $X$ gives rise to
an element $[X]$ of $\Ext_{A(D)}^1(S_{\Fil},S_{\C})$, by Theorem
\ref{t-extscharact}. In addition, we see that $X\in
Z_{\C}(S^*_{\Fil})$ if and only if $\xi\in H^\infty(\Fil)$, while
$X\in Z_{\C}(S_{\Fil})$ if and only if there exists another
holomorphic function $\eta$ with the property that $\xi+\ol{\eta}\in
L^\infty (\Fil)$.

\section{Vanishing of the polynomial subgroup in the case of contractions}
The goal of this section is to show that
$\Ext_{\rm{poly}}^1(T,S_{\E}^*)=0$ whenever $T$ is a contraction. To
achieve it, we make use of the functional model of a completely
non-unitary contraction which we briefly recall (see \cite{nagy2010}
or \cite{bercovici1988} for greater detail).

Let $\E,\E_*$ be separable Hilbert spaces and let $\Theta\in
H^\infty(B(\E,\E_*))$ be a contractive (weakly) holomorphic
function. Define $\Delta\in L^\infty(B(\E))$ as follows
$$
\Delta(e^{it})=\sqrt{I-\Theta(e^{it})^*\Theta(e^{it})}.
$$
If we set
$$
M(\Theta)=\{\Theta u\oplus \Delta u: u\in H^2(\E) \},
$$
then the space $H(\Theta)$ is defined as
$$
H(\Theta)=(H^2(\E_*)\oplus \ol{\Delta L^2(\E)})\ominus M({\Theta})
$$
and we have
$$
S(\Theta)=P_{H(\Theta)}(S\oplus U)|H(\Theta)
$$
where $S=S_{\E_*}$  is the unilateral shift on $H^2(\E_*)$ and $U$
is the unitary operator of multiplication by the variable $e^{it}$
on $L^2(\E)$.

In order to proceed, we require two technical lemmas which are most
likely well-known. We provide the calculation for the reader's
convenience.

\begin{lemma}\label{l-projfunctmodel}
Let $e_*\in \E_*$. Then,
$$
P_{M(\Theta)}(e_*\oplus 0)=\Theta\Theta(0)^*e_*\oplus
\Delta\Theta(0)^*e_*
$$
and
$$
P_{H(\Theta)}(e_*\oplus 0)=(I-\Theta\Theta(0)^*e_*)\oplus
\left(-\Delta\Theta(0)^*e_*\right).
$$
\end{lemma}
\begin{proof}
For any $u\in H^2(\E)$ we have
$$
\langle e_*\oplus 0,\Theta u\oplus \Delta u \rangle_{H^2(\E_*)\oplus
L^2(\E)}=\langle \Theta(0)^*e_*,u(0) \rangle_{\E}
$$
and
\begin{align*}
&\langle \Theta\Theta(0)^*e_*\oplus \Delta\Theta(0)^*e_*, \Theta
u\oplus \Delta u\rangle_{H^2(\E_*)\oplus L^2(\E)} \\
&=\langle
\Theta^*\Theta\Theta(0)^*e_*, u\rangle_{H^2(\E)}+\langle
\Delta^2\Theta(0)^*e_*,u \rangle_{L^2(\E)}\\
&=\langle \Theta^*\Theta\Theta(0)^*e_*, u\rangle_{L^2(\E)}+\langle
(I-\Theta^*\Theta)\Theta(0)^*e_*,u \rangle_{L^2(\E)}\\
&=\langle \Theta(0)^*e_*,u\rangle_{L^2(\E)}\\
&=\langle \Theta(0)^*e_*,u(0)\rangle_{\E}
\end{align*}
which shows the first equality, and the second follows immediately.
\end{proof}

\begin{lemma}\label{l-rangeST}
The range of the operator $S(\Theta)$ is
$$
\{ f_1\oplus f_2\in H(\Theta):f_1(0)\in \Theta(0)\E\}.
$$
\end{lemma}
\begin{proof}
Assume that
$$
f_1\oplus f_2=S(\Theta)(v_1\oplus v_2)
$$
for $v_1\oplus v_2\in H(\Theta)$. Then, we can write
$$
f_1\oplus f_2=zv_1\oplus e^{it} v_2+\Theta u\oplus \Delta u
$$
for some $u\in H^2(\E)$, and therefore
$$
f_1(0)=\Theta(0)u(0)
$$
lies in the range of $\Theta(0)$.

Conversely, pick $f=f_1\oplus f_2\in H(\Theta)$ such that
$f_1(0)=\Theta(0)e$ for some $e\in \E$. Then, the function
$$
f_1-\Theta e\in H^2(\E_*) $$ vanishes at $z=0$, so we can find
another function $v_1\in H^2(\E_*)$ with the property that
$$
f_1-\Theta e=zv_1.
$$
Since $U\Delta=\Delta U$, we find that the function
$$
v_2=U^*(f_2-\Delta e)
$$
lies in $\ol{\Delta L^2(\E)}$ and satisfies
$$
Uv_2=f_2-\Delta e.
$$
We see that
\begin{align*}
P_{H(\Theta)}(Sv_1\oplus Uv_2)&=P_{H(\Theta)}((f_1-\Theta e)\oplus
(f_2-\Delta e))\\
&=P_{H(\Theta)} (f_1\oplus f_2)\\
&=f_1\oplus f_2
\end{align*}
and therefore
$$
f_1\oplus f_2=S(\Theta)P_{H(\Theta)}(v_1\oplus v_2)
$$
lies in the range of $S(\Theta)$.
\end{proof}

The following is the crucial technical step in the proof of the main
result.

\begin{theorem}\label{t-ZEfunctmodel}
Let $\Fil, \Fil_*, \E$ be separable Hilbert spaces. Let $\Theta\in
H^\infty(B(\Fil,\Fil_*))$ be a contractive holomorphic function.
Then,
$$
B(H(\Theta),\E)=B(H(\Theta),\E)S(\Theta)^*+Z_{\E}(S(\Theta)^*).
$$
\end{theorem}
\begin{proof}
Let $X\in B(H(\Theta),\E)$. Define $X_1:H(\Theta)\to \E$ as
$$
X_1h=XP_{H(\Theta)}P_{\Fil^*\oplus \{0\}}\widehat{h}(0)
$$
where for a function $h\in H(\Theta)$ we define $\widehat{h}(n)$ to
be its $n$-th Fourier coefficient, which lies in $\Fil_*\oplus
\Fil$. Given $e\in \E$ and $h=h_1\oplus h_2\in H(\Theta)$, we have
\begin{align*}
\left\langle X_1 h, e\right\rangle_{\E}
&=\left\langle XP_{H(\Theta)}P_{\Fil^*\oplus \{0\}}\widehat{h}(0),e\right\rangle_{\E}\\
&=\left\langle h_1(0)\oplus 0,X^*e\right\rangle_{H^2(\Fil_*)\oplus L^2(\Fil)}\\
&= \left\langle \widehat{h}(0),P_{\Fil_*\oplus \{0\}}\widehat{X^* e}(0)\right\rangle_{\Fil_*\oplus \Fil}\\
&=\left\langle h,P_{H(\Theta)}P_{\Fil_*\oplus
\{0\}}\widehat{X^*e}(0)\right\rangle_{H(\Theta)}
\end{align*}
whence
$$
X_1^*e=P_{H(\Theta)} P_{\Fil_*\oplus \{0\}}\widehat{X^*e}(0) .
$$
Set $X_2=X-X_1$ and $\widehat{X^*e}(0)=f_*\oplus f\in \Fil_*\oplus
\Fil$. Using Lemma \ref{l-projfunctmodel}, we find
$$
X_1^*e=(I-\Theta \Theta(0)^*)f_*\oplus \left(-\Delta \Theta(0)^*
f_*\right).
$$
A straightforward verification using Lemma \ref{l-rangeST}
establishes that the range of $X_2^*$ is contained in the range of
$S(\Theta)$. By Douglas's Lemma, this implies in turn that
$$
X_2\in B(H(\Theta),\E)S(\Theta)^*.
$$
Since $X=X_1+X_2$, it remains only to check that $X_1\in
Z_{\E}(S(\Theta)^*)$. First, we note that for $h=h_1\oplus h_2\in
H(\Theta)$ we have
$$
S(\Theta)^{*n}h=(P_{H^2(\Fil_*)}\ol{z}^n h_1)\oplus e^{-int}h_2
$$
and the Fourier coefficient of order zero of $S(\Theta)^{*n}h$ is
therefore equal to $\widehat{h}(n)$. Consequently,
$$
X_1S(\Theta)^{*n}h=XP_{H(\Theta)}P_{\Fil_*\oplus
\{0\}}\widehat{h}(n)
$$
and
$$
\sum_{n=0}^\infty \|X_1S(\Theta)^{*n}h\|^2\leq \|X\|^2
\sum_{n=0}^\infty \|\widehat{h}(n)\|^2 \leq \|X\|^2 \|h\|^2
$$
so that $X_1\in Z_{\E}(S(\Theta)^*)$.
\end{proof}

We now come to the main result of the paper (recall Definition
\ref{d-poly}).

\begin{theorem}\label{t-extcontraction}
Let $\E$ be a separable Hilbert space and let $S_{\E}:H^2(\E)\to
H^2(\E)$ be the unilateral shift. Then,
$\Ext^1_{\rm{poly}}(T,S_{\E}^*)=0$ for every operator $T$ which is
similar to a contraction.
\end{theorem}
\begin{proof}
Since extension groups are invariant under similarity, we may assume
that $T\in B(\hil)$ is a contraction. Then, it is well-known that
there exists a reducing subspace $M\subset \hil$ with the property
that $T|M$ is completely non-unitary and $T|M^\perp$ is unitary.
According to this decomposition, it is easy to verify that any
bounded operator $X:\hil\to H^2(\E)$ giving rise to an element
$[X]\in \Ext^1_{\rm{poly}}(T,S_{\E}^*)$ can be written as
$X=(X_1,X_2)$, where $[X_1]\in \Ext^1_{\rm{poly}}(T|M,S_{\E}^*)$ and
$[X_2]\in \Ext^1_{\rm{poly}}(T|M^\perp,S_{\E}^*)$. Using Theorem
\ref{thmunitarysplit} we see that $[X]=0$ if and only if $[X_1]=0$.
Therefore, we may assume that $T$ (and hence $T^*$) is completely
non-unitary. By Theorem VI.2.3 of \cite{nagy2010}, we know that
$T^*$ is unitarily equivalent to $S(\Theta)$ for some contractive
operator-valued holomorphic function $\Theta$, so for our purposes
we may as well take $T^*$ to be equal to $S(\Theta)$. In light of
Theorem \ref{t-ZEfunctmodel}, we find
$$
B(H(\Theta),\E)=B(H(\Theta),\E)S(\Theta)^*+Z_{\E}(S(\Theta)^*)
$$
and thus an application of Corollary \ref{c-ext0crit} completes the
proof.
\end{proof}

Theorem \ref{t-isomproj} and Theorem \ref{t-extcontraction}
illustrate a clear difference between $S_{\E}$ and $S_{\E}^*$ on the
level of extension groups: $\Ext_{\rm{poly}}^1(T,S_{\E}^*)=0$ for
every contraction $T$ while $\Ext_{A(\D)}^1(T,S_{\E})=0$ only when
the contraction $T$ is similar to an isometry. The reader will
object immediately to the fact that we are considering the
polynomial subgroup in one case and the full group in the other. In
some sense however, there is no discrepancy between the two
settings. Indeed, by Theorem \ref{t-extscharact}, every element in
$\Ext^1_{A(\D)}(T,S_{\E})$ can be represented by an operator
$X:\hil\to H^2(\E)$ with range contained in $\E$. In particular, we
see that $S_{\E}^*X=0$, and thus $X$ is a polynomial operator.
Therefore, the group $\Ext_{A(\D)}^1(\cdot,S_{\E})$ coincides with
$\Ext_{\rm{poly}}^1(\cdot,S_{\E})$. This is not the case if $S_{\E}$
is replaced by $S_{\E}^*$, as is shown in Section 6.

\section{The case of non-contractive modules: Pisier's counterexample}
Much of the vanishing results for extension groups obtained in the
previous sections focus on extensions of the unilateral shift (and
its adjoint) by contractive modules. It is natural to wonder what
happens for extensions by polynomially bounded operators which are
not similar to a contraction. Unfortunately, few examples of such
operators are known. In fact, only the family of counterexamples
introduced by Pisier in \cite{pisier1997} is available. Let us
recall the details of his construction here.

Let $S_{\Fil}:H^2(\Fil)\to H^2(\Fil)$ be the unilateral shift with infinite
multiplicity, where
$$
\Fil=\bigoplus_{n=1}^\infty(\C^2)^{\otimes n}.
$$
Define
$$
V=\left(
\begin{array}{cc}
1 & 0\\
0 & -1
\end{array}
\right)
$$
and
$$
D=\left(
\begin{array}{cc}
0 & 0\\
1 & 0
\end{array}
\right).
$$
For $0\leq k \leq n-1$, define $C_{k,n}:(\C^2)^{\otimes n} \to
(\C^2)^{\otimes n}$ as
$$
C_{k,n}=V^{\otimes (k+1)}\otimes D \otimes I^{\otimes( n-k-2)}
$$
and for any $k\geq 0$ set
$$
W_k=\bigoplus_{n= k+1}^\infty C_{k,n}
$$
which acts on $\Fil$. It is well-known (see \cite{davidson1997} or
\cite{paulsen2002}) that the sequence of operators $\{W_k\}_k\subset
B(\Fil)$ satisfies the so-called \emph{canonical anticommutation
relations}. Given a sequence $\alpha=\{\alpha_n\}_{n=0}^\infty \subset \C$, we
define a Hankel operator $X_{\alpha}$ acting on $H^2(\Fil)$ by
$$
X_{\alpha}=(\alpha_{i+j}W_{i+j})_{i,j=0}^\infty
$$
and we set
$$
R(X_{\alpha})=\left(\begin{array}{cc}
S_{\Fil}^* & X_{\alpha} \\
0 & S_{\Fil}
\end{array} \right).
$$
The following result can be found in \cite{davidson1997} and
\cite{ricard2002}.

\begin{theorem}
The operator $R(X_{\alpha})$ is polynomially bounded if and only if
$$
\sup_{k\geq 0}(k+1)^2\sum_{i= k}^\infty|\alpha_i|^2
$$
is finite, and it is similar to a contraction if and only if
$$
\sum_{k=0}^\infty (k+1)^2|\alpha_k|^2
$$
is finite.
\end{theorem}

We noted at the end of Section 5 that
$$
\Ext^1_{A(\D)}(\cdot,S_{\E})=\Ext^1_{\rm{poly}}(\cdot,S_{\E}).
$$
However, things are different for $\Ext^1_{A(\D)}(\cdot,S_{\E}^*)$.
Indeed, if the sequence $\alpha=\{\alpha_n\}_n$ is chosen such that
$R(X_{\alpha})$ is polynomially bounded but not similar to a
contraction, then $[X_{\alpha}]$ is a non-trivial element of
$\Ext_{A(\D)}^1(S_{\Fil},S_{\Fil}^*)$. In particular, Theorem
\ref{t-extcontraction} implies that
$$
\Ext_{\rm{poly}}^1(S_{\Fil},S_{\Fil}^*) \neq
\Ext_{A(\D)}^1(S_{\Fil},S_{\Fil}^*).
$$
We do not know whether equality holds if we require that the shift be of
finite multiplicity.

The remainder of this section is dedicated to the study of the group
$$
\Ext_{\rm{poly}}^1(R(X_{\alpha}),S_{\C})=\Ext_{A(\D)}^1(R(X_{\alpha}),S_{\C}).
$$
Of particular interest is the case where $R(X_{\alpha})$ is not
similar to a contraction, which lies outside the reach of Theorem
\ref{t-isomproj} where little is known.

We start by giving an alternative formulation of Corollary
\ref{c-ext0crit} adapted to the unilateral shift of multiplicity
one. For a Hilbert module $(\hil, T)$, define $Z(T)\subset \hil$ to
be the set consisting of those vectors $x\in \hil$ with the property
that there exists a constant $c_x>0$ such that
$$
\sum_{n=0}^\infty |\langle h,T^{*n} x\rangle|\leq c_x \|h\|^2
$$
for every $h\in \hil$.

\begin{lemma}\label{l-ext0critmult1}
Let $(\hil,T)$ be a Hilbert module and let $S_{\C}:H^2\to H^2$ be
the unilateral shift with multiplicity one. Then
$$
T^*\hil+Z(T)=\hil
$$
if and only if
$$
\Ext^1_{\rm{poly}}(T,S_{\C}^*)=0.
$$
\end{lemma}
\begin{proof}
Note that any operator $X:\hil\to \C$ is given by $Xh=\langle
h,\xi\rangle$ for some $\xi\in \hil$. It is a routine verification to
establish that under this identification, the equality
$$
B(\hil,\C)T+ Z_{\C}(T)= B(\hil,\C)
$$
corresponds to
$$
T^{*}\hil+Z(T)=\hil
$$
so the result follows from Corollary \ref{c-ext0crit}.
\end{proof}

This corollary offers the advantage over the more complicated
general version that the equality we are interested in takes places
inside the Hilbert space $\hil$ instead of inside the Banach space
$B(\hil,\E)$. Note also that the discussion at the end of Section 4
shows that $Z(S_{\Fil}^*)=H^\infty(\Fil)$. We now state a simple
result.

\begin{lemma}\label{l-kernel}
Let $(\hil, T)$ be a Hilbert module. Any operator $X\in Z_{\E}(T)$
for which $[X]=0$ in $\Ext_{A(\D)}^1(T,S_{\E})$ belongs to
$B(\hil,\E)T$.
\end{lemma}
\begin{proof}
If $[X]=0$, then $X=S_{\E}L-LT$ and arguing as in the proof of
Corollary \ref{c-bddbelow} we find that $X=L'T$.
\end{proof}

Let us now apply this lemma to the study of
$\Ext_{\rm{poly}}^1(R(X_{\alpha}),S_{\C})$. Using the fact that
$S_{\Fil}^*X=XS_{\Fil}$, it is readily verified that
$$
R(X_\alpha)^{*n}=\left(\begin{array}{cc}
S_{\Fil}^n &  0\\
nX^*S_{\Fil}^{n-1} & S_{\Fil}^{*n}
\end{array} \right)
$$
for every integer $n\geq 1$. Thus, for $h\in H^2(\Fil)$ we have that
$h\oplus 0 \in Z(R(X_\alpha))$ if and only if
$$
\sum_{n=1}^\infty \left|\left\langle \left(
\begin{array}{c}
z^n h \\
nX^* z^{n-1}h
\end{array}
\right),
 \left(
 \begin{array}{c}
  g_1 \\
  g_2
\end{array}
\right) \right\rangle\right|^2 \leq c \|g\|^2
$$
for some constant $c>0$ and every $g=g_1\oplus g_2\in
H^2(\Fil)\oplus H^2(\Fil)$. Consequently, $h\oplus 0 \in
Z(R(X_\alpha))$ is equivalent to $h\in Z(S_{\Fil}^*)$ and
$$
\sum_{n=1}^\infty |\langle  nX^* z^{n-1}h, g\rangle|^2\leq c \|g\|^2
$$
for every $g\in H^2(\Fil)$. Notice at this point that for $\omega\in
\Fil$ we have
$$
X^* z^{n}\omega=\sum_{m=0}^\infty z^m \ol{\alpha_{m+n}}
W^{*(m+n)}\omega.
$$
Let $\omega=e_1\oplus 0\oplus 0 \oplus \ldots\in \Fil$ where
$e_1=(1,0)\in \C^2$. Then, $W^*_k \omega=0$ for $k\geq 1$ so that
$$
nX^* z^{n-1}\omega=0
$$
for $n\geq 2$ and thus
$$
\sum_{n=1}^\infty |\langle  nX^* z^{n-1}\omega, g\rangle|^2\leq
\|X^*\omega\|^2 \|g\|^2
$$
for every $g\in H^2(\Fil)$. In addition, it is clear that $\omega\in
H^\infty(\Fil)=Z(S^*_{\Fil})$ so in fact $\omega\in
Z(R(X_{\alpha}))$.

Define now $\Omega:H^2(\Fil)\oplus H^2(\Fil)\to \C$ by
$$
\Omega (f_1\oplus f_2)=\langle f_1(0), \omega \rangle_{\Fil}.
$$
Since $\omega\in Z(R(X_{\alpha}))$, we have that $\Omega\in Z_{\C}(R(X_{\alpha}))$, whence
$$
[\Omega]\in \Ext_{A(\D)}^1(R(X_{\alpha}),S_{\C})
$$
by Theorem \ref{t-extscharact}. Moreover,
 $\Omega (\omega\oplus 0)=1$ and
$R(X_{\alpha})(\omega\oplus 0)=0$ so that
$$
\Omega \ker R(X_{\alpha})\neq 0
$$
and thus $[\Omega]\neq 0$ in $\Ext_{A(\D)}^1(R(X_{\alpha}),S_{\C})$
by Lemma \ref{l-kernel}. In other words,
$$
\Ext_{A(\D)}^1(R(X_{\alpha}),S_{\C})\neq 0.
$$
It is easy to see that this argument can be adapted to show that
$$
\Ext_{A(\D)}^1(R(X_{\alpha}),S_{\E})\neq  0
$$
for every separable Hilbert space $\E$. Note finally that
$[\Omega]=0$ in $\Ext_{A(\D)}^1(R(X_{\alpha}),S^*_{\C})$ by Theorem
\ref{t-innercrit}.

In conclusion, let us mention that the question of whether or not
$$\Ext_{\rm{poly}}^1(R(X_{\alpha}),S^*_{\C})$$ vanishes (in the case
where $R(X_{\alpha})$ is not similar to a contraction, of course)
remains open. Given its direct relation to the projectivity of the
unilateral shift of multiplicity one, this problem is obviously
meaningful. We hope that Lemma \ref{l-ext0critmult1} may help settle
it in the future.

\bibliography{biblio}

\def\cprime{$'$} \def\cprime{$'$} \def\cprime{$'$} \def\cprime{$'$}
% \bib, bibdiv, biblist are defined by the amsrefs package.
\begin{bibdiv}
\begin{biblist}

\bib{bercovici1988}{book}{
      author={Bercovici, Hari},
       title={Operator theory and arithmetic in {$H^\infty$}},
      series={Mathematical Surveys and Monographs},
   publisher={American Mathematical Society},
     address={Providence, RI},
        date={1988},
      volume={26},
        ISBN={0-8218-1528-8},
      review={\MR{954383 (90e:47001)}},
}

\bib{carlson1995}{article}{
      author={Carlson, Jon~F.},
      author={Clark, Douglas~N.},
       title={Cohomology and extensions of {H}ilbert modules},
        date={1995},
        ISSN={0022-1236},
     journal={J. Funct. Anal.},
      volume={128},
      number={2},
       pages={278\ndash 306},
         url={http://dx.doi.org/10.1006/jfan.1995.1032},
      review={\MR{1319957 (96a:46129)}},
}

\bib{carlson1997}{article}{
      author={Carlson, Jon~F.},
      author={Clark, Douglas~N.},
       title={Projectivity and extensions of {H}ilbert modules over {$\bold
  A(\bold D^N)$}},
        date={1997},
        ISSN={0026-2285},
     journal={Michigan Math. J.},
      volume={44},
      number={2},
       pages={365\ndash 373},
         url={http://dx.doi.org/10.1307/mmj/1029005711},
      review={\MR{1460421 (98i:46047)}},
}

\bib{carlson1994}{article}{
      author={Carlson, Jon~F.},
      author={Clark, Douglas~N.},
      author={Foias, Ciprian},
      author={Williams, J.~P.},
       title={Projective {H}ilbert {$\bold A(\bold D)$}-modules},
        date={1994/95},
        ISSN={1076-9803},
     journal={New York J. Math.},
      volume={1},
       pages={26\ndash 38, electronic},
         url={http://nyjm.albany.edu:8000/j/1994/1_26.html},
      review={\MR{1284181 (95g:46136)}},
}

\bib{chen2000}{article}{
      author={Chen, Xiaoman},
      author={Guo, Kunyu},
       title={Cohomology and extensions of hypo-\v {S}ilov modules over unit
  modulus algebras},
        date={2000},
        ISSN={0379-4024},
     journal={J. Operator Theory},
      volume={43},
      number={1},
       pages={69\ndash 81},
      review={\MR{1740895 (2001f:46073)}},
}

\bib{clancy1998}{article}{
      author={Clancy, Robert~S.},
      author={McCullough, Scott},
       title={Projective modules and {H}ilbert spaces with a
  {N}evanlinna-{P}ick kernel},
        date={1998},
        ISSN={0002-9939},
     journal={Proc. Amer. Math. Soc.},
      volume={126},
      number={11},
       pages={3299\ndash 3305},
         url={http://dx.doi.org/10.1090/S0002-9939-98-04624-3},
      review={\MR{1473659 (99a:47014)}},
}

\bib{davidson1997}{article}{
      author={Davidson, Kenneth~R.},
      author={Paulsen, Vern~I.},
       title={Polynomially bounded operators},
        date={1997},
        ISSN={0075-4102},
     journal={J. Reine Angew. Math.},
      volume={487},
       pages={153\ndash 170},
      review={\MR{1454263 (98d:47003)}},
}

\bib{didas2006}{article}{
      author={Didas, Michael},
      author={Eschmeier, J{\"o}rg},
       title={Unitary extensions of {H}ilbert {$A(D)$}-modules split},
        date={2006},
        ISSN={0022-1236},
     journal={J. Funct. Anal.},
      volume={238},
      number={2},
       pages={565\ndash 577},
         url={http://dx.doi.org/10.1016/j.jfa.2006.02.002},
      review={\MR{2253733 (2007i:46047)}},
}

\bib{douglas1989}{book}{
      author={Douglas, Ronald~G.},
      author={Paulsen, Vern~I.},
       title={Hilbert modules over function algebras},
      series={Pitman Research Notes in Mathematics Series},
   publisher={Longman Scientific \& Technical},
     address={Harlow},
        date={1989},
      volume={217},
        ISBN={0-582-04796-X},
      review={\MR{1028546 (91g:46084)}},
}

\bib{ferguson1997}{article}{
      author={Ferguson, Sarah~H.},
       title={Backward shift invariant operator ranges},
        date={1997},
        ISSN={0022-1236},
     journal={J. Funct. Anal.},
      volume={150},
      number={2},
       pages={526\ndash 543},
         url={http://dx.doi.org/10.1006/jfan.1997.3132},
      review={\MR{1479551 (98j:47017)}},
}

\bib{gulinskiy2003}{article}{
      author={Gulinskiy, T.~O.},
       title={Projectivity of {H}ilbert modules over uniform algebras},
        date={2003},
        ISSN={0305-0041},
     journal={Math. Proc. Cambridge Philos. Soc.},
      volume={135},
      number={2},
       pages={269\ndash 275},
         url={http://dx.doi.org/10.1017/S0305004103006868},
      review={\MR{2006064 (2004m:46140)}},
}

\bib{guo1999}{article}{
      author={Guo, Kunyu},
       title={Normal {H}ilbert modules over the ball algebra {$A(B)$}},
        date={1999},
        ISSN={0039-3223},
     journal={Studia Math.},
      volume={135},
      number={1},
       pages={1\ndash 12},
      review={\MR{1686367 (2000h:46065)}},
}

\bib{muhly1995}{article}{
      author={Muhly, Paul~S.},
      author={Solel, Baruch},
       title={Hilbert modules over operator algebras},
        date={1995},
        ISSN={0065-9266},
     journal={Mem. Amer. Math. Soc.},
      volume={117},
      number={559},
       pages={viii+53},
      review={\MR{1271693 (96c:47060)}},
}

\bib{page1970}{article}{
      author={Page, Lavon~B.},
       title={Bounded and compact vectorial {H}ankel operators},
        date={1970},
        ISSN={0002-9947},
     journal={Trans. Amer. Math. Soc.},
      volume={150},
       pages={529\ndash 539},
      review={\MR{0273449 (42 \#8327)}},
}

\bib{paulsen2002}{book}{
      author={Paulsen, Vern},
       title={Completely bounded maps and operator algebras},
      series={Cambridge Studies in Advanced Mathematics},
   publisher={Cambridge University Press},
     address={Cambridge},
        date={2002},
      volume={78},
        ISBN={0-521-81669-6},
      review={\MR{1976867 (2004c:46118)}},
}

\bib{pisier1997}{article}{
      author={Pisier, Gilles},
       title={A polynomially bounded operator on {H}ilbert space which is not
  similar to a contraction},
        date={1997},
        ISSN={0894-0347},
     journal={J. Amer. Math. Soc.},
      volume={10},
      number={2},
       pages={351\ndash 369},
         url={http://dx.doi.org/10.1090/S0894-0347-97-00227-0},
      review={\MR{1415321 (97f:47002)}},
}

\bib{ricard2002}{article}{
      author={Ricard, {\'E}ric},
       title={On a question of {D}avidson and {P}aulsen},
        date={2002},
        ISSN={0022-1236},
     journal={J. Funct. Anal.},
      volume={192},
      number={1},
       pages={283\ndash 294},
         url={http://dx.doi.org/10.1006/jfan.2001.3899},
      review={\MR{1918497 (2003f:47013)}},
}

\bib{nagy2010}{book}{
      author={Sz.-Nagy, B{\'e}la},
      author={Foias, Ciprian},
      author={Bercovici, Hari},
      author={K{\'e}rchy, L{\'a}szl{\'o}},
       title={Harmonic analysis of operators on {H}ilbert space},
     edition={Second},
      series={Universitext},
   publisher={Springer},
     address={New York},
        date={2010},
        ISBN={978-1-4419-6093-1},
         url={http://dx.doi.org/10.1007/978-1-4419-6094-8},
      review={\MR{2760647 (2012b:47001)}},
}

\end{biblist}
\end{bibdiv}
%\bibliography{/home/raphael/Documents/FichiersTex/biblio}
\bibliographystyle{plain}

\end{document}